\documentclass[12pt,a4paper]{article}
\usepackage{amssymb}
\usepackage[latin1]{inputenc}
\usepackage[english]{babel}
\usepackage{amsmath}
\usepackage{amsthm}
\usepackage{verbatim}

\newtheorem{theorem}{Theorem}

\newtheorem{remark}{Remark}
\newtheorem{remarks}[remark]{Remarks}
\newtheorem{example}{Example}
\newtheorem{examples}[example]{Examples}
\newtheorem{proposition}{Proposition}
\newtheorem{definition}{Definition}
\newtheorem{corollary}{Corollary}

\newcommand{\RR}{\mathbb{R}}

\newcommand{\cC}{\mathcal{C}}

\newcommand{\HM}{\mathrm{HM}}
\newcommand{\HCM}{\mathrm{HCM}}
\newcommand{\GGC}{\mathrm{GGC}}

\begin{document}
\title{A class of scale mixtures of
gamma(k)-distributions that are generalized gamma convolutions}
\author{Anita Behme\thanks{Technische Universit\"at
      M\"unchen, Center for Mathematical Sciences, Boltzmannstrasse 3, D-85748 Garching,  Germany, \texttt{behme@ma.tum.de}, phone: +49-89-289-17424, fax: +49-89-289-17435. }$\,$
 and Lennart Bondesson\thanks{Department of Mathematics and Mathematical Statistics, Ume\aa\ University, SE-90187 Ume\aa, Sweden, \texttt{lennart.bondesson$@$math.umu.se}.}}
\date{\today}
\maketitle

\vspace{-1cm}
\begin{abstract}
Let $ k >0 $ be an integer and  $ Y $ a standard Gamma$(k)$
distributed random variable. Let $ X $ be an independent positive
random variable with a density that is hyperbolically monotone
(HM) of order $ k.$   Then $Y\cdot X$ and $Y/X $ both have
distributions that are generalized gamma convolutions ($\GGC$s).
This result extends a result of Roynette et al. from 2009 who
treated the case $ k=1 $ but without use of the $\HM$-concept.
Applications in excursion theory of diffusions and in the theory
of exponential functionals of L\'evy processes are mentioned.
\end{abstract}



\section{Introduction}\label{secintro} 

A generalized gamma convolution ($\GGC$) is a limit distribution for
sums of independent gamma distributed random variables (rvs). The
$\GGC$s were introduced by the actuary O. Thorin in 1977 when he
tried to prove that the lognormal distribution is infinitely
divisible (see (\cite{Thorin})). He used a technique that later
on led the second author of this paper to introduce in
\cite{LB92} the concept of hyperbolic complete monotonicity
($\HCM$). The simpler concept of hyperbolic monotonicity (HM) was
mentioned in \cite[pp.\ 101-102]{LB92} and more carefully studied
in \cite{LB97}.

The $\GGC$s have got applications in many different fields including
infinite divisibility (e.g.\ Steutel and van Harn
\cite{SteutelHarn}), mathematical analysis (e.g.\ Schilling et
al.\ \cite{SchillingSongVondracek}), stochastic processes (e.g.\
James et al.\ \cite{James-et-al} and Behme et al.\
\cite{BehmeMaejimaMatsuiSakuma}), and financial mathematics (e.g.\
Barndorff-Nielsen et al.\ \cite{B-N}).

In 2009 Roynette et al.\ \cite{RoynetteValloisYor} proved a novel
$\GGC$ result that has provided stimulus to the present work. In our
terminology, they showed that the product of an exponentially
distributed rv $ Y $ and another independent rv $ X $ has a $\GGC$
distribution provided that the density of $ X $ is $\HM$. We will
give a new and more transparent proof of this result and
generalize it considerably to cover gamma distributions.

The paper is organized as follows. In Section~\ref{secbackground}
the $\HM$, $\HCM$ and $\GGC$ theory is briefly recalled. In
Section~\ref{secMainResult} the main result that the product of a
gamma variable with shape parameter $ k $ and an rv with $\HM_k$
distribution has a $\GGC$ distribution is given. This result can be
formulated in several alternative ways. It has also an  important
extension.  The proof contains some surprising elements.
Applications, analytical as well as stochastic process related ones,
are given in Section~\ref{secappl}. Finally, in
Section~\ref{secfinal} some open problems are mentioned.

\section{Background} \label{secbackground}

Basic facts on hyperbolic monotonicity
(HM) and generalized gamma convolutions ($\GGC$s) are presented here.
They are taken from Bondesson \cite{LB92, LB97}. Much information about
$\GGC$s and hyperbolic complete monotonicity ($\HCM$) can also be found
in the book by Steutel and van Harn \cite{SteutelHarn}.

\subsection{Hyperbolic monotonicity}\label{secHCM}

Let $ f $ be a nonnegative function on $(0,\infty)$. Consider,
for any fixed $ u>0,$ the function $f(uv)f(u/v)$, $v>0$.
Obviously it is invariant under the transformation $v \mapsto
v^{-1}. $ It follows that it is a function $h(w)$ of $w=v +
v^{-1} $ since the value of $ w $ determines the set
$\{v,v^{-1}\}.$

 \begin{definition} \rm A nonnegative function $f$ on $(0, \infty)$ is said to be
 {\sl hyperbolically monotone} ($\HM$ or $\HM_1$) if, for each fixed $u>0, $ the
 function $ h(w) =f(uv)f(u/v) $ is non-increasing as a function of $
 w=v+v^{-1}.$ More generally, it  is called {\sl hyperbolically
 monotone of order $ k $} ($\HM_k$) if $ (-1)^j  h^{(j)}(w) \ge 0, \;
 j=0,1,\dots,k-1 $ and $ (-1)^{k-1}h^{(k-1)}(w) $ is non-increasing.
 If this holds for all $k \ge 1,$  $ f $ is also  called
 {\sl hyperbolically completely monotone} ($\HCM$).
\end{definition}

 The class of $\HM_k$-functions is also denoted $\HM_k$.
 Obviously $$ \HCM= \HM_\infty  \subset \dots
 \subset
  \HM_3 \subset \HM_2 \subset \HM_1 = \HM.
  $$ Simple examples of $\HCM$-functions are provided by (with $\gamma \in \mathbb
 R, c\ge 0):$
 $ x^\gamma, e^{-cx}, $ and $  e^{-c/x}.$ It is apparent that the $\HM_k$-class is closed with respect to
 multiplication of functions. For  $ f \in \textrm{HM}_k, $
 obviously $ f(uv)f(u/v) \le (f(u))^2.$  It easily follows that $ \log f(e^x) $
 is concave and hence that $ f(x) \le C x^\gamma $ for some
 constants $ C\ge 0 $ and $\gamma \in \mathbb R $ (depending on $
 f).$
 Every $\HM_k$-function $ f $ can therefore
 be modified to an $\HM_k$ probability density function (pdf) by
 multiplication by a factor $ \exp(-\delta_1 x-\delta_2 x^{-1})$
 (with $\delta_1 >0 $ and $ \delta_2 >0 $
 arbitrarily small) and
 a normalizing constant. In this paper we are mainly concerned
 with pdfs.

\begin{example} \rm Let $ f $ be a pdf on $ (a,b) \subset (0,\infty)
$ of the form $ f(x)= C (x-a)^{\alpha-1}(b-x)^{\beta -1},$ where $
C $ is a constant. It can be shown  that for $ \alpha \ge 1 $ and
$ \beta \ge 1,  f $ is $\HM_k$ for $ k= \min([\alpha],[\beta]),$
where $[\cdot] $ denotes integer part. However, if $ a=0,$ then
for any value of $\alpha,\, f $ is $\HM_k $ for $ k= [\beta].$ In
particular, the $U(a,b)$ density is $\HM_1.$ In fact, in this case
it is easy to see that $ h(w) = f(uv)f(u/v) $ is 0 for all $u$
sufficiently large or small and that for the other values of $ u,
\, h(w)$ equals  1 if $ w $ is below some bound and otherwise 0.
$ \quad \Box$
\end{example}

\begin{example}\label{exampleuniform}
\rm Let $ X = U_1U_2\cdots U_k, $ where the random variables
(rvs) $U_i$ are independent and uniformly distributed on $(0,1)$.
Since $ - \log X $ has a $\mathrm{Gamma}(k,1)$-distribution,  $
f_X(x)=\frac{1}{(k-1)!}(-\log x)^{k-1}, 0<x<1. $ This pdf is
$\HM_k.$ In fact, $ h(w)=f(uv)f(u/v) =0 $ for $ u \ge 1,$ whereas,
for $ u<1,$  $ h(w) \propto ((\log u)^2- (\log v)^2))^{k-1} $ if
$ u<v<u^{-1} $ (i.e.\ if $ 2\le w < u+u^{-1}$) and otherwise
vanishes. The $\HM_k$ result then follows from the fact that $ d
(\log v)^2/dw $ is completely monotone (CM). In fact, this
derivative can be shown to be equal to $ \int_0^\infty
(1+t^2+tw)^{-1}dt. \qquad \Box $
\end{example}


The following result, which concerns powers, products and ratios
of rvs, is important. Its proof (in \cite{LB97}) is far from
trivial. A main idea in the proof is to use hyperbolic
substitutions of the form $ x=uv, y=u/v $ in certain double
integrals.

 \begin{proposition}\label{prop:HMkproducts} Let $ X $ and $ Y $ be independent rvs with $\HM_k$-densities $(X \sim \HM_k, Y \sim \HM_k$).
 Then, for any $ q \in \mathbb R $ with $|q|\ge 1,$ we have $ X^q
 \sim \HM_k$. Moreover, $ X \!\cdot\! Y  \sim \HM_k $ and
 $ X/Y \sim \HM_k$.
\end{proposition}

A simple consequence of Proposition \ref{prop:HMkproducts} (with
one of the rvs exponentially distributed) is that the Laplace
transform of an $\HM_k$ function is $\HM_k.$ Let $ X $ have the
$\HM_2$-density $ f(x) =2\max(0,1-x) $ and let $Y \! \sim \!
U(0,1)$  (with an $\HM_1$-density). Then it can be shown that $
X/Y \not \sim \HM_2.$ Thus there is no trivial extension of
Proposition \ref{prop:HMkproducts}.

\medskip
The $\HM_1$-densities can be identified as follows (see \cite{LB97}).

\begin{proposition} \label{prop:HM1logconcave} We have $ X \sim \HM_1$ if and only if
$ Y=\log X $ has a pdf that is logconcave, i.e. $\log f_Y(y) $ is
concave. Equivalently, $ X \sim \HM_1$ if and only if  $ f_X(x)=
C \exp(-\int_{x_0}^x \frac{\psi(y)}{y}dy ),$ where $ \psi $ is
non-decreasing,  $ C $ a constant, and $ x_0 $ is suitable chosen.
\end{proposition}

With this, the well-known fact that logconcavity is preserved under
convolution (see e.g. \cite[pp. 17-23]{Dharmadhikari}) becomes a simple
consequence of Proposition~\ref{prop:HMkproducts} for $ k=1. $


\medskip
Typical $\HCM$ (=$\HM_\infty) $ pdfs have the form $ f(x) = C
x^{\beta-1}\prod_{i=1}^n (1+c_ix)^{-\gamma_i},$ where the
parameters are positive, or are limits of such densities. In fact,
all $\HCM$-densities (and functions) are such limits. An open
problem is to find canonical representations for $\HM_k$-densities
for $ 1 < k <\infty.$

\medskip
The $\HM_k$-class of densities (functions) can alternatively be
described by the condition that
\begin{equation} \label{repr}
h(w) = f(uv)f(u/v) = c_u + \int_{(w,\infty)} (\lambda-w)^{k-1}
H_u(d\lambda),
\end{equation}
where $ c_u \ge 0 $ and $ H_u(d\lambda) $ is a nonnegative
measure. The simple example $ f(x)=x^\gamma $ gives $c_u =
u^{2\gamma} $ and $ H_u(d\lambda)\equiv 0.$ However, for a pdf we
must have $ c_u=0.$  The representation (\ref{repr}) follows from
a representation of the non-increasing  function $
(-1)^{k-1}h^{(k-1)}(w) $ as an integral over $ (w,\infty)$ (or
possibly  $[w,\infty)$)  of a nonnegative measure. For instance,
for $ k=2 $ we put $-h'(w) =  \int {\bf 1} ( w
<\lambda)H_u(d\lambda). $ We then get, by a change of the order
of integration,  $$ h(w)- h(\infty)= - \! \int_w^\infty \!
h'(\tilde w)d\tilde w =   \int \! \! \int \! {\bf 1}(w< \tilde
w<\lambda)d\tilde w H_u(d\lambda) = \int_{(w,\infty)}\!\!\!\!\!
(\lambda-w)H_u(d\lambda).
$$

The  representation \eqref{repr} was derived and used in \cite{LB97}. For
functions with monotone derivatives up to some order it seems to
have been first used by Williamson \cite{Williamson}.

\subsection{Generalized gamma convolutions}\label{secGGC}

Convolving different gamma distributions, $\mathrm{Gamma}(u,t)$, with pdfs
and Laplace transforms (LTs) of the forms $ f(x)=
(\Gamma(u))^{-1}x^{u-1}t^u \exp(-xt) $ and $\phi(s)= (\frac{t}{t+s})^u, $
respectively, and then taking weak limits, Thorin \cite{Thorin} was led to
the following definition.

\begin{definition}\rm A {\sl generalized gamma convolution ($\GGC$)} is
a probability distribution on $ [0,\infty) $ with LT of the form
$$ \phi(s) = \exp\left(-as + \int_{(0,\infty)} \log\left(\frac{t}{t+s}\right)U(dt)\right),
$$ where (the left-extremity) $ a\ge 0 $ and $ U(dt) $ is a nonnegative measure on
$ (0,\infty)$ (with finite mass for any compact subset of $ (0,\infty)$) such that  $\int_{(0,1)} |\log \, t|U(dt) <\infty $ and $
\int_{(1,\infty)} t^{-1}U(dt) < \infty. $
\end{definition}

The $\GGC$-class of distributions is closed with respect to (wrt)
addition of independent random variables and wrt weak limits.
Each $\GGC$ is infinitely divisible and each convolution root of a
$\GGC$ is a $\GGC$ as well. The pdf $ f(x) $ of a $\GGC$ is strictly
positive on $(a,\infty)$ and, if $ a=0 $ and
$\beta=\int_{(0,\infty)} U(dt) $ is finite, then $ f(x)=
x^{\beta-1}h(x), $ where $ h(x) $ is completely monotone
(see \cite[p. 49]{LB92}).

\medskip
The pdf of a $\GGC$ need not be $\HM_1.$  For instance, for a gamma
distribution with shape parameter less than 1 and shifted to have
left-extremity $ a
>0 $ the pdf is not $\HM_1$. An $\HM_k$-density, which may have compact
support, is in general not a $\GGC$. However (see \cite[Theorem
5.1.2]{LB92}):

\begin{proposition} \label{prop:HCMinGGC} If the pdf $ f $ on $ (0,\infty) $ is
$\HCM$, then it is a $\GGC$. Thus $\HCM \subset \GGC$.
\end{proposition}

Many well-known pdfs are $\HCM$ and therefore also $\GGC$s and hence
infinitely divisible. For instance gamma densities are $\HCM$. Then
it follows from Proposition \ref{prop:HMkproducts} (for $
k=\infty $)  that also the power $ q, \, q \ge 1,$ of the ratio
of two independent gamma variables has a density that is $\HCM$.
This density is of the form $ f(x)= C
x^{\beta-1}(1+cx^\alpha)^{-\gamma}, x>0,$ with $ \alpha=q^{-1}.$
Every lognormal density is also $\HCM$.

\medskip
The next proposition gives a  characterization of the LT of a $\GGC$
(\cite[Theorem 6.1.1]{LB92}).

\begin{proposition} \label{prop:GGCLTisHCM} A function $\phi(s) $ on $(0,\infty)$ is
the LT of a $\GGC$ if and only if $\phi(0+)=1 $ and $ \phi $ is $\HCM$.
\end{proposition}

This result will be our basic tool in Section 4. Since the LT of
an $\HM_k$ function is $\HM_k$, and this also holds for $
k=\infty,$ Proposition~\ref{prop:HCMinGGC} can be seen as a
consequence of Proposition~\ref{prop:GGCLTisHCM}. Using another complex characterization of the LT of a GGC, we can get the following result (\cite[Theorem
4.2.1]{LB92}).

\begin{proposition} \label{prop:GammaoverlogconcaveisGGC} Let $ Y \sim  \mathrm{Gamma}(1,1)$ and let $ X
>0 $ be an independent rv with a density $ f(x) $ that is
logconcave (or only such that $ xf(x) $ is logconcave). Then $
Y/X \sim \GGC$.
\end{proposition}

One should notice that in Proposition
\ref{prop:GammaoverlogconcaveisGGC} the rv $ X $ is not assumed
to have an $\HM_1$-density.

\begin{proposition} \label{LTtimesf} If $ f(x) $ is the density
of a $\GGC$ and $ x^{-\alpha}f(x), $ where $\alpha \ge 0, $ can be
normalized to be the pdf $ g(x) $ of a probability distribution,
then $ g(x) $ is also the pdf of a $\GGC$. \end{proposition}

This result is only a limit case of \cite[Theorem 6.2.4]{LB92}.
The following recent result from \cite{LB14} needs to be
mentioned. It can be proved by the help of Proposition
\ref{prop:GGCLTisHCM}.

\begin{proposition} \label{prop:GGCproducts} Let $ X \sim \GGC$ and $ Y \sim \GGC$ be
independent rvs. Then $ X \cdot Y \sim \GGC$.
\end{proposition}

Well-known examples of $\GGC$ distributions include the log-normal
distribution and positive strictly $\alpha$-stable distributions.
Also, each {\it negative} power of a gamma variable is shown to
have a $\GGC$-distribution in \cite{BoschSimon}. Bosch and Simon \cite{BS} and Jedidi and Simon \cite{JedidiSimon} give other novel results on $\HM$, $\HCM$, and $\GGC$
distributions.

\section{Main result} \label{secMainResult}

Here the main result is presented as a theorem in Section
\ref{secmaintheorem}. Moreover comments are given. The proof is
presented in Section \ref{secproof}.

\subsection{Formulation of the main result and
comments}\label{secmaintheorem}

\begin{theorem}\label{maintheorem} Let $ k \ge 1 $ be an integer. Let $ Y \sim
 \mathrm{Gamma}(k,1)$ and $ X \sim \HM_k$ be independent rvs. Then $ Y
\! \cdot \! X \sim \GGC$ and $ Y/X \sim \GGC$.
\end{theorem}

We give some comments on the above theorem.

\begin{remarks}\rm \label{remarksmain}
\begin{enumerate}
 \item For $ k=1 $  Theorem~\ref{maintheorem} differs from
 Proposition~\ref{prop:GammaoverlogconcaveisGGC}.
  One should notice that $ X \sim \HM_1 \Leftrightarrow X^{-1}  \sim \HM_1$
 but logconcavity of $ f_X $ is not equivalent to logconcavity of $ f_{X^{-1}}.
 $ One can also notice that every gamma density is $\HCM$ (and thus
 $\HM_1$) but only logconcave when the shape parameter is $\ge 1.$


\item  In the case $k=1$ the LT $\phi_1(s)= \int_0^\infty (x+s)^{-1}xf_X(x)dx $ of $Y/X$ for independent $Y \sim \mathrm{Gamma}(1,1)$ and
$X \sim \HM_1$ is the Stieltjes transform (or double Laplace transform) of the measure $x f_X(x)dx$.
For $k>1$ the LT $\phi_k(s)= \int_0^\infty (x+s)^{-k}x^k f_X(x)dx $ coincides with the so-called generalized Stieltjes
transform (of order $k$) of the measure $x^k f_X(x) dx$. In that sense the above theorem can be restated as follows:
Assume $f_X(x)$ is an $\HM_k$ function. Then the $k$-th order generalized Stieltjes transform of $x^k f_X(x) dx$ is $\HCM$, i.e. it is the LT of a $\GGC$.

\item Clearly Theorem~\ref{maintheorem} remains true if $Y\sim \mathrm{Gamma}(k, \theta)$ for any $\theta>0$, since in this case $\theta Y\sim \mathrm{Gamma}(k,
1)$. Considering $ Y \sim \mathrm{Gamma}(k,k)$  and letting $ k
\rightarrow \infty $ we get that $ Y \rightarrow 1 $ in
probability. Hence for $ X \sim \mathrm{HM}_k $ with $ k $ fixed
it is neccessary in the theorem to have a restriction upwards on
the shape parameter of $ Y $ since otherwise it would incorrectly
follow that $ \HM_k \subset \GGC$. For instance,
if $ Y \sim \mathrm{Gamma}(2,1)$ and $ X \sim U(1,2), $ then $ f_X $ is
$\HM_1$ but $ Y/X \not \sim \GGC$.

\item Letting again $ k \rightarrow \infty $ and so that  $Y \rightarrow 1 $ in
probability, we get back Proposition~\ref{prop:HCMinGGC} as a
limit case of Theorem~\ref{maintheorem}.
 Since a
$\mathrm{Gamma}(k,1)$ density is $\HCM$, it also follows that the class of
$\GGC$s provided by Theorem~\ref{maintheorem} is closed wrt
multiplication and division of independent rvs. However, if $
Z=Y\!\cdot\!X $ with $ Y \sim \mathrm{Gamma}(k,1)$ and  $ X \sim \HM_k, $ it is not true that $ Z^{-1} $ always has the same
representation.

\item Theorem~\ref{maintheorem} can also be expressed in the following way. Any
scale mixture of $\mathrm{Gamma}(k)$ distributions with a scale mixing
$\HM_k$-density is a $\GGC$. It is well known (\cite[Theorem 3.3, p.\
334]{SteutelHarn}) that any scale mixture of $\mathrm{Gamma}(1)$
distributions is infinitely divisible (ID). More generally, any
scale mixture of $\mathrm{Gamma}(2)$ distributions is ID
(\cite{Kristiansen}). However, for $ k>2 $ ID fails to hold in
general for such mixtures.



\end{enumerate}

\end{remarks}

There is a nice extension of Theorem~\ref{maintheorem} which we
see as a corollary of it.

\begin{corollary} \label{corollary}  Let $ Y \sim \mathrm{Gamma}(r,1) $ be independent of $ X \sim \HM_k $ where
$ r >0 $ and $ k $ is an integer such that $ k\ge r. $ Then $
Y\!\cdot\!X \sim \GGC$  and $ Y/X \sim \GGC$.
\end{corollary}

\begin{proof}  Since $ X \sim \HM_k$ if and only if $1/X \sim \HM_k$, it suffices to consider the ratio $Z=Y/X.$ Let $
\alpha=k-r$ and let $ Y' \sim $ Gamma$(k,1).$ Then
$$ f_Z(z)=\int_0^\infty xf_Y(zx)f_X(x)dx = \frac{1}{\Gamma(r)}
\int_0^\infty x(zx)^{r-1}e^{-zx}f_X(x)dx  $$ $$ = \frac{\Gamma(k)}{\Gamma(r)}z^{-\alpha} \int_0^\infty x f_{Y'}(zx)x^{-\alpha}f_X(x)dx. $$ Here $ x^{-\alpha}f_X(x) $ is
$\HM_k$ and so is, for any $\delta\ge 0, \,
x^{-\alpha}e^{-\delta/x}f_X(x). $ Letting if necessary $ \delta
> 0 $ and normalizing this latter function to become the pdf of an
rv $ X',$ we get from Theorem~\ref{maintheorem} that $ Y'/X' \sim \GGC$. Using then  Proposition~\ref{LTtimesf} and letting $\delta
\rightarrow 0, $ we conclude that $ Y/X \sim \GGC$.
\end{proof}

\subsection{Proof of the main result}\label{secproof}

The proof of Theorem~\ref{maintheorem} is given in two parts.
First the case $ k=1 $ is treated. This proof is short but
contains the essential ideas. The proof in the general case
becomes more technical. Of course, we use the $\HCM$-characterization
of the LT of a $\GGC$ and hyperbolic substitutions in the proofs.
For the transformation $ T=t+t^{-1}, $ we avoid to use the
inverse transformation $ t= T/2 \pm \sqrt{T^2/4-1}.$ In fact, the
$\HCM$-concept was introduced in the early 1990s in order to avoid,
at least in presentations, such inverse transformations.

\medskip \noindent
{\it Proof of Theorem~\ref{maintheorem}, k=1}. It suffices to
consider the ratio $Y/X,$ where $ Y \sim \mathrm{Gamma}(1,1)$. The LT $
\phi(s) $ of the distribution of the ratio is given by, with $
f=f_X,$
$$ \phi(s) = E(\exp(-sY/X)) = \int_0^\infty E(\exp(-sY/x))f(x)dx
= \int_0^\infty \frac{x}{x+s}f(x)dx. $$ For fixed $ s>0,$ consider
$$ J=\phi(st)\phi(s/t) = \int_0^\infty \int_0^\infty \frac{xy}{(x+st)(y+ \,s/t)} f(x)f(y)dx dy. $$ In view of Proposition~\ref{prop:GGCLTisHCM}, we
only have to show that $ J $ is completely monotone (CM) wrt $
T=t+t^{-1}.$ We make the hyperbolic substitution $ x=uv, y=u/v $
with Jacobian with modulus $2u/v$. Hence
$$ J= \int_0^\infty  \int_0^\infty  \frac{2u}{v} \frac{u^2}{(uv+st)(u/v\, + \, s/t)}
f(uv)f(u/v) dudv. $$ Using the representation $ f(uv)f(u/v)=
\int_{[w,\infty)} H_u(d\lambda), $
where $ H_u(d\lambda) $ is a nonnegative measure and $
w=v+v^{-1}, $ letting $ b=b(\lambda)\ge 1 $ be such that $
b+b^{-1} =\lambda,$ letting $a=u/s, $ and changing the order of
integration, we get by some simple algebra that
$$ J=  \int_0^\infty \frac{2u^2}{s} \int_2^\infty \underbrace{\left(
\int_{1/b}^b \frac{t}{(v + \, t/a)(v+at)} dv\right)}_{=:J_1}
H_u(d\lambda)du. $$ It is now evident that it suffices to show
that for each $ b\ge 1 $ and each $ a>0 $ the interior
$v$-integral $J_1$ is CM wrt to $ T=t+t^{-1}.$ For $ b=1, $ $
J_1=0,$ so it suffices to consider the case $ b>1.$ The integral
$ J_1 $ is a function of $ T $ since the change $ t \mapsto
t^{-1} $ leaves  $ J_1 $ invariant which is shown by the
substitution $ v =1/v'.$ Now $ J_1 $ can be calculated
explicitly. In fact, by a partial fraction expansion we have for
$ a\ne 1, $
$$ \frac{t}{(v + \, t/a)(v+at)} = \frac{1}{a-a^{-1}} \left(\frac{1}{v + \, t/a} -
\frac{1}{v+at}\right) $$ and hence, for $ a\ne 1,$ by an integration
and some simplification,
$$ J_1 = \frac{1}{a-a^{-1}} \log \left( \frac{(t+ab)(t+ (ab)^{-1})}{(t +\, a/b)(t+\, b/a)}\right) = \frac{1}{a-a^{-1}} \log \frac{T+A}{T+B },$$  where $A=ab+ (ab)^{-1}, B= a/b+\, b/a.$ For $
a=1,$ $ J_1= (b-b^{-1})/(T+b+b^{-1}). $ Since $a\mapsto a^{-1}$ leaves $J_1$ unchanged, we may without restriction assume that $ a>1 $ (and as earlier $ b>1$), and then $ A>B $
and $ J_1
>0.$ Moreover, we get that the $k$-th derivative of $ J_1$, i.e.
here its first derivative, has the form
$$ \frac{dJ_1}{dT} = \frac{1}{a-a^{-1}} \left(\frac{1}{T+A}-\frac{1}{T+B}\right) = \frac{1}{a-a^{-1}}\: \frac{B-A}{(T+A)(T+B)}$$ and
this derivative is negative. Since $ (T+A)^{-1}(T+B)^{-1} $ is
CM, we get as desired that $ (-1)^j J_1^{(j)}(T) \ge 0, \:
j=0,1,2,\dots, $ and the proof is complete. $\Box $

\medskip
We now proceed with the general proof of Theorem~\ref{maintheorem} for any
integer $ k \ge 1.$ We shall see that the above proof needs some
complementary arguments.

\medskip \noindent
{\it Proof of Theorem~\ref{maintheorem}, general $k$}. Let $ Y \sim \mathrm{Gamma}(k,1)$
and $ X \sim \HM_k$ be independent. Then the LT $\phi(s) $ of
the distribution of $ Y/X $ is given by
$$ \phi(s)= \int_0^\infty \left(\frac{x}{x+s}\right)^k f(x)dx. $$
Hence using \eqref{repr}
\begin{eqnarray*}
  J&=& \phi(st)\phi(s/t) = \int_0^\infty \int_0^\infty \frac{x^ky^k}{(x+st)^k(y+ \,s/t)^k} f(x)f(y)dx dy\\
&=& \int_0^\infty  \int_0^\infty \frac{2u}{v}
\frac{u^{2k}}{(uv+st)^k (u/v + \, s/t)^k} f(uv)f(u/v)dudv \\
&=&\int_0^\infty  \int_0^\infty \frac{2u}{v}
\frac{u^{2k}}{(uv+st)^k (u/v + \, s/t)^k} \int_w^\infty
(\lambda-w)^{k-1}H_u(d\lambda) du dv,
\end{eqnarray*}
where $ w=v+v^{-1} $ and
$ H_u(d\lambda) $ is a nonnegative measure. Again we let $ b =
b(\lambda) \ge 1 $ be such that $ b+b^{-1}=\lambda $ and  put $
a=u/s. $ After a change of the order of integration with the $
v-$integral as the inner integral and noticing that $ b+b^{-1} -
v -v^{-1} = (b-v)(v-b^{-1})/v,$  we see by some algebraic
manipulations that it suffices to show that the integral
\begin{eqnarray} \label{J1}
J_k = \int_{1/b}^b I_k dv,\quad  \textrm{where}\quad I_k= \frac{t^k
((b-v)(v-b^{-1}))^{k-1}}{(v+\,t/a)^k (v+at)^k},
\end{eqnarray}
is CM wrt $ T=t+t^{-1}.$ The same argumentation as in the case $
k=1 $ shows that $ J_k $ is a function of $ T.$ An important fact
is that $ J_k $ can be calculated explicitly for all integers $ k
$ although $ J_k $ becomes complicated for $ k $ large. Since $
I_k $ is a rational function of $ v, $ we can get an expression
for $ J_k $  by using first a partial fraction expansion of $ I_k
$ wrt $ v. $ However, it is more efficient to use an alternating
generating function:
$$ GF(z)= \sum_{k=1}^\infty (-z)^{k-1}J_k =
\int_{1/b}^b \sum_{k=1}^\infty (-z)^{k-1} I_k dv $$ $$ =
\int_{1/b}^b \frac{1}{(v+\, t/a)(v+at)/t\,+\,z(b-v)(v-b^{-1})}dv.
$$ Minimizing over $v$ and $t,$ we see that the series is absolutely convergent at least if $|z|\le \frac{a}{(1+a)^2}\frac{(b-1)^2}{b}.$  Since the denominator in the integrand is a
quadratic function of $ v $ and as such a function can be
factorized into  two real linear factors for $ z\ge 0, $ we get by
integration and considerable simplification with the notation $
\alpha= a+a^{-1}, \beta = b+b^{-1}$ that
$$ GF(z)= \frac{1}{\sqrt{\Delta}}\log R, \quad \textrm{where} $$
$$ \Delta= (\alpha+ \beta z)^2-4-4z^2+4zT \quad \textrm{and} \quad R= \frac{T-2z +
\frac{1}{2}\beta (\alpha + \beta z) + \frac{1}{2}(b-b^{-1})\sqrt{\Delta}}{T-2z + \frac{1}{2}\beta (\alpha + \beta z) - \frac{1}{2}(b-b^{-1})\sqrt{\Delta}}\,. $$  It is far
from obvious but some calculation shows that
$$ \frac{d}{dz}\log R = \frac{2(b-b^{-1})}{\sqrt{\Delta}}. $$
In fact, the product of the numerator and the denominator in $ R
$ does not depend on $ z $ so the derivative above is just twice
the derivative of the logarithm of the numerator in $ R. $ Of
course, $ J_k= J_k(T)=\frac{(-1)^{k-1}}{(k-1)!} GF^{(k-1)}(0).$
Now it is  not difficult to see that  with, as earlier, $ A=
ab+(ab)^{-1} $ and $ B=a/b+b/a, $ we get
\begin{eqnarray} \label{ImpExp}
J_k(T)= P_k(T) + Q_k(T)\log\left(\frac{T + A}{T+B }\right),
\end{eqnarray}
where $ P_k(T) $ and $ Q_k(T) $ are polynomials in T of degrees $
k-2 $ and $ k-1, $ respectively. For $ k=1, \, P_k(T) $ vanishes.
For $ k=1,2,$ and 3, we have
$$\begin{array}{ccc}
  P_1(T)=0, &P_2(T)= -2\frac{b-b^{-1}}{(a-a^{-1})^2}, &P_3(T)=
-3 \frac{b-b^{-1}}{(a-a^{-1})^4} (2T + A+B),  \\
Q_1(T)= \frac{1}{a-a^{-1}}, &Q_2(T)= \frac{2T+A+B}{(a-a^{-1})^3},
&Q_3(T)= \frac{ 6T^2+ 6(A+B)T+(A+B)^2+2AB}{(a-a^{-1})^5}.
  \end{array}
$$
By using the above expressions for $ P_k(T) $ and $ Q_k(T) $ one
can easily verify that at least for $ k=1,2,3 $ we have somewhat
surprisingly
\begin{eqnarray} \label{dJ1dT} \frac{d^k\!J_k}{dT^k} = (-1)^k (k-1)! \frac{(b-b^{-1})^{2k-1}}{(T+A)^k(T+B)^k} .\end{eqnarray}
To see that (\ref{dJ1dT}) is
completely general, some additional argumentation is needed.
Since $ P_k(T) $ has degree $k-2,$ it has no influence at all on
the $k$-th derivative of $ J_k. $ Since $ Q_k(T) $ has degree $
k-1 $ and hence $ Q_k^{(k)}(T) \equiv 0, $ it also follows from
(\ref{ImpExp}) after some reflection that
\begin{eqnarray} \label{dJ1dTMORE} \frac{d^k\!J_k}{dT^k} = \frac{R_k(T)}{(T+A)^k (T+B)^k},
\end{eqnarray}  where $ R_k(T) $ is a polynomial of degree at
most $2k-1.$  To see that really $ R_k(T) $ is a constant, $
(-1)^k (k-1)!(b-b^{-1})^{2k-1}, $ we look at the case when $ t
\rightarrow  \infty. $ Then $ T = t+t^{-1} $ is very close to $
t.$ From (\ref{J1}) we get that
$$ J_k \sim \frac{1}{t^k}
\int_{1/b}^b \big((b-v)(v-b^{-1})\big)^{k-1}dv \; \sim \;
B(k,k)\frac{(b-b^{-1})^{2k-1}}{T^k} \; \: \textrm{as} \,\; t
\rightarrow \infty, $$ where $ B(k,k) = (k-1)!(k-1)!/(2k-1)!$.
Since $ k(k+1)\cdots(2k-1)B(k,k)= (k-1)!, $ we also get that
$$ \frac{d^k\!J_k}{dT^k} \; \sim \; (-1)^k(k-1)! \frac{(b-b^{-1})^{2k-1}}{T^{2k}}
\;\;
\textrm{as}\;\: t \rightarrow \infty. $$
 Since $ R_k(T) $ in
(\ref{dJ1dTMORE}) is a polynomial, this asymptotic relation can
only hold when the polynomial is a constant and hence
(\ref{dJ1dT}) holds for all $ k$. To complete the proof, we use
that $(T+A)^{-k}(T+B)^{-k} $ is CM and hence $(-1)^{j}J_k^{(j)}(T)
\ge 0 $ for $ j=k,k+1,\dots.$ Then it only remains to verify that
these inequalities also hold for $ j=0,1,\dots,k-1.$ Using the
same argumentation as above we have that, for each $ j \ge 0,$  $
J_k^{(j)}(T) = O(T^{-k-j})$ as $ T \rightarrow \infty. $ In
particular $ J_k^{(j)}(T) $ vanishes at $ T=\infty. $ It follows
that
$$ J_k^{(j)}(T) = -\int_T^\infty J_k^{(j+1)}(\tilde T)\,d\tilde T, \;\; j=0,1,2,\dots $$ We see
that the sign of $ J_k^{(k-1)}(T) $ is opposite to that of $
J_k^{(k)}(T).$ The same then holds for the sign of
$J_k^{(k-2)}(T) $ compared with that of $ J_k^{(k-1)}(T), $ etc.
This shows that $ J_k(T) $  is CM wrt $ T $ as desired. $\Box$

\medskip
\begin{remark}\rm In the general case some technical
details have been omitted in the proof. However, it is easy to
check all statements by using a program for symbolic algebra. In
fact, the simple form in \eqref{dJ1dT} for the derivative $
J_k^{(k)}(T) $ was discovered in that way.
\end{remark}

\section{Applications}\label{secappl}

Theorem~\ref{maintheorem} and Corollary \ref{corollary} have a
wide range of possible applications. We will discuss a few in
this section.

\subsection{Excursion theory}

The random process foundations for the research carried out in this article have
been laid by Roynette et al.\ \cite{RoynetteValloisYor} and
Salminen et al.\ \cite{SalminenValloisYor}. In these articles the
authors study excursion times of recurrent linear diffusions on
$\RR_+$. More precisely, given an $\RR_+$-valued recurrent
diffusion $(X_t)_{t\geq 0}$ and defining the last and the next
visit in $0$ via
$$g_t:=\sup \{s\leq t, X_s =0 \}, \quad d_t:=\inf \{s\geq t, X_s =0 \}$$
they are interested in the rvs
\begin{equation}\label{extimes}
 Y^{(1)}_p= Z_p- g_{Z_p}, \quad Y^{(2)}_p= d_{Z_p}-Z_p, \quad Y^{(3)}_p= d_{Z_p}- g_{Z_p},
\end{equation}
where $Z_p$ denotes an exponential rv with density $ p e^{-pz},
z>0,$ independent of $(X_t)_{t\geq 0}$. In
\cite{SalminenValloisYor} it is shown, that all $Y^{(i)}$ are
infinitely divisible, while in \cite{RoynetteValloisYor} the
authors give conditions for $Y^{(i)}$ to have $\GGC$ distributions.
These conditions are stated in terms of the Krein measure of the
L\'evy measure of the inverse local time at $0$ of $(X_t)_{t\geq
0}$.

For their proof of the $\GGC$ property, Roynette et al.\ first show,
for $k=1,$ a reformulation of Theorem \ref{maintheorem}
(\cite[Theorem 2]{RoynetteValloisYor}). They do not use the
HM-concept but define a class $\cC$ of functions which
essentially coincides with the class $\HM_1$. The proof of
\cite[Theorem 2]{RoynetteValloisYor} then relies on the
$\HCM$-characterization of the LT of a $\GGC$
(Proposition~\ref{prop:GGCLTisHCM}). Although also our proof for
$ k=1 $ uses Proposition~\ref{prop:GGCLTisHCM} it is shorter than
theirs because of our use of a suitable hyperbolic substitution
in a double integral and the avoidance of certain inverse
transformations.

Further in \cite{RoynetteValloisYor} the LTs of the $Y^{(i)}$s
are shown to be Stieltjes transforms of measures whose densities are $\HM_1$ (compare with
Remark~\ref{remarksmain}(ii)).

\medskip
In the following we will indicate via an example how one can also
use our main theorem in the case $k=2$ to prove the $\GGC$ property
of $Y^{(3)}_p$ as defined in \eqref{extimes}. Therefore we
briefly recall some notation from \cite{SalminenValloisYor} and
\cite{RoynetteValloisYor}.

 Let $(L_t)_{t\geq 0}$ be the
continuous local-time of $(X_t)_{t\geq 0}$ at $0$ and
$(\tau_u)_{u\geq 0}$ its right-continuous inverse. Then
$(\tau_u)_{u\geq 0}$ is a subordinator and as such has a L\'evy
exponent $\psi$ and a L\'evy density $\nu$, i.e.
$$E[e^{-\lambda \tau_u}]= e^{-u\psi(\lambda)}= \exp \left(-u \int_0^\infty (1-e^{-\lambda x}) \nu(x) dx \right) $$
where further $\nu$ has the Krein representation
$$ \nu(x)=\int_0^\infty e^{-xz} K(dz)$$
with Krein measure $K$ of $\nu$.

\begin{proposition}
Assume that the Krein measure $K$ is such that the function $f(u)$ defined via
$$f(u^{-1})=\int_{(u-p)\vee 0}^{u} K(dz)$$
is an $\HM_2$ function. Then $Y_p^{(3)} \sim \GGC$.
\end{proposition}
\begin{proof}
It was shown in \cite{SalminenValloisYor}, that the distribution
of  $Y^{(3)}_p$ is a $\mathrm{Gamma}(2)$-mixture. In particular it can be
deduced from \cite[Equations (46), (49) and
(50)]{SalminenValloisYor} that the density of $Y^{(3)}_p$ is
given by
$$f_{Y^{(3)}_p}(u)= \frac{1}{\psi(p)} \int_0^\infty  u e^{-ux}  \int_{x-p}^x K(dz)  dx$$
which shows that $Y^{(3)}_p=Y\cdot X$ where $Y\sim \mathrm{Gamma}(2,1)$ and $X$ is independent of $Y$ with density $f_X$ defined via $f_X(u^{-1}) = \frac{1}{\psi(p)} \int_{u-p}^u K(dz).$ 
Thus the claim follows from Theorem \ref{maintheorem} in the case $k=2$.
\end{proof}

\subsection{Exponential functionals of L\'evy processes}

Let $\xi=(\xi_t)_{t\geq 0}$ be a L\'evy process such that
$\xi_t\to -\infty$ as $t\to \infty$. Then the exponential
functional of $\xi$ is defined as
$$I_\xi:=\int_{(0,\infty)} e^{\xi_t}dt.$$
Such exponential functionals appear as stationary distributions
of generalized Ornstein-Uhlenbeck processes and they have
attracted a lot of interest throughout the last years (see e.g.\
\cite{carmonapetityor}, the survey paper \cite{bertoinyor} or the
more recent contributions \cite{BehmeLindner,
BehmeMaejimaMatsuiSakuma, PardoPatieSavov} to name just very few
references).

It is known, that $I_\xi\sim \GGC$ in several cases. E.g.\ Dufresne
(e.g.\ \cite[Equation (16)]{bertoinyor}) showed that $I_\xi
\overset{d}= \frac{2}{\sigma^2} G^{-1}_{2a/\sigma^2}$ where
$G_\gamma\sim\mathrm{Gamma}(\gamma,1)$, whenever $\xi$ is a Brownian
motion with variance $\sigma^2$ and drift $a<0$. Concerning
processes $\xi$ with jumps, one has for example the following
Proposition.

\begin{proposition}
 Suppose that $\xi$ is a compound Poisson process, i.e. $\xi_t=\sum_{i=1}^{N_t}X_i$
with i.i.d. jump heights $X_i,\ i=1,2,\ldots$, such that $-\infty<E[X_1]<0$ and $e^{X_1}\sim \GGC$. Then $I_\xi \sim \GGC$.
\end{proposition}
\begin{proof}
 The proof can be carried out along the lines of the proof of \cite[Proposition 3.2]{BehmeMaejimaMatsuiSakuma} using the more recent Proposition \ref{prop:GGCproducts}.
\end{proof}

Still, assuming that $\xi_t=at-N_t, t\geq 0,$ for $a<0$ and a subordinator (i.e. a nondecreasing L\'evy process) $(N_t)_{t\geq 0}$, one easily observes that $I_\xi$ has bounded support and therefore cannot be infinitely divisible such that in particular $I_\xi \not\sim \GGC$.

In \cite{PardoPatieSavov}, based on the Wiener-Hopf factorization
of L\'evy processes, the authors obtain factorizations of exponential
functionals. In particular, in case of a spectrally negative process $\xi$
with $\xi_t \to -\infty$, they prove that
$$I_\xi \overset{d}= \frac{I_H}{G_\gamma},$$
where  $H = (H_t)_{t\geq 0}$ is the descending ladder height
process of $\xi$ and $G_\gamma\sim \mathrm{Gamma}(\gamma,1)$, with
$\gamma$ depending on the characteristics of $\xi,$ is independent
of $H$. We refer to \cite{Sato} or \cite{Bertoin} for any further
information on L\'evy processes, their characteristics and their
Wiener-Hopf factorizations.

\medskip
Since $-H=(-H_t)_{t\geq 0}$ is a subordinator with drift $a_H$
and L\'evy jump measure $\nu_H$, say, it follows from \cite[Example
B]{carmonapetityor} that if $H$ is non-trivial, then $I_H$ admits
a density $f(s)$ which fulfills the integro-differential equation
$$(1-a_H s)f(s)= \int_s^\infty \bar{\nu}_H(\log(t/s))f(t) dt,$$
where $\bar{\nu}_H(x)=\nu_H((x, \infty))$.
In particular, if $\bar{\nu}_H(s)=ce^{-bs}$, $b,c>0$, and
$a_H>0$, the authors prove that
$$I_H \overset{d}= \frac{1}{a_H} Z_{b+1, c/a_H},$$
where $Z_{\alpha,\beta}$ is a Beta rv on $(0,1)$ with parameters
$\alpha,\beta>0$. Hence $I_H\sim\HM_k$ for $k\leq \min([b+1],
[c/a_H])$.

Now by Corollary~\ref{corollary}, $I_\xi$ is the reciprocal of a
$\GGC$ if $k\geq \gamma$. Notice that in general such inverses of $\GGC$s are not $\GGC$s themselves. However, in this case $I_\xi\sim \HM_k$, since $G_\gamma\sim \HCM$. 

Conversely,  if again $\bar{\nu}_H(s)=ce^{-bs}$, $b,c>0$, but $a_H=0$, then $I_H$ itself is Gamma distributed and so $I_\xi\sim \GGC$.

\subsection{Constructing $\GGC$s}

Using Theorem \ref{maintheorem}  we can construct explicit densities and LTs of $\GGC$s as we shall do in the following.

\begin{examples} \rm
\begin{enumerate}
 \item Let $ Y \sim \mathrm{Gamma}(1,1)$  and $ X = U \sim U(0,1)$ be independent.
 Then we have the following LTs  and pdfs for $ YU $
and  $Y/U$, respectively:
\begin{eqnarray*}
 \phi_{YU}(s) = \frac{\log(1+s)}{s},& \quad& \phi_{Y/U}(s)  = 1+ s\log(\frac{s}{1+s}),\\
f_{YU}(x)= Ei(x)=
\int_1^\infty y^{-1}e^{-yx}dy,& \quad &f_{Y/U}(x)= \frac{1}{x^2}(1-(1+x)e^{-x}).
\end{eqnarray*}
By Example \ref{exampleuniform} we have $U\sim \HM_1$ and hence by Theorem~\ref{maintheorem} the above pdfs are $\GGC$s and the LTs are $\HCM$.
\item Now let $ Y \sim \mathrm{Gamma}(2,1)$ and $X=\min(U_1,U_2),$ with $ U_1, U_2 \sim U(0,1)$ independent and independent of $Y$. Then $ f_X(x)=2(1-x), 0<x<1,$ which belongs to $\HM_2$.
We can also represent $ X $ as $ X \overset{d}= U_1U_2^{1/2}.$ We
get the following LTs and pdfs:
\begin{eqnarray*}
  \phi_{YX}(s)= \frac{2}{s}(1-\frac{\log(1+s)}{s}), &\quad&\phi_{Y/X}(s)= 1+6s+
(6s^2+4s)\log(\frac{s}{1+s}),\\
 f_{YX}(x)=2e^{-x} - 2xEi(x), &\quad & f_{Y/X}(x)= \frac{1}{x^3}(-12+4x+(2x^2+8x+12)e^{-x}).
\end{eqnarray*}
Again by Theorem~\ref{maintheorem} the pdfs
are $\GGC$s and the LTs are $\HCM$.
\end{enumerate}

\end{examples}

Many similar examples can be obtained from
Corollary~\ref{corollary} by letting $ Y \sim $ Gamma$(r,1)$ with
a real $ r. $ 

\section{Final comments}\label{secfinal}

There are reasons to believe that Theorem~\ref{maintheorem} (as
well as Proposition \ref{prop:HMkproducts}) can be extended to
cover the case that $ k $ is any real number $ \ge 1. $ Maybe it
can even be extended to all real $ k>0.$  As a definition of an
$\HM_k $ function in the real case, the integral representation
(\ref{repr}) can be used.  For any real $j$ and $k$ such that $ 0
< j< k $ we have $\HM_k \subset \HM_j.$ To see this, one can use
(\ref{repr}) together with the simple formula
$$ (\lambda-w)^{k-1} = \frac{1}{B(j,k-j)} \int_w^\lambda
(\tilde \lambda-w)^{j-1} (\lambda-\tilde \lambda)^{k-j-1} d\tilde
\lambda\,.  $$  For $ k\ge 1 $ the $\HM_k$ class is closed wrt
multiplication of functions. However, it is not closed for $k<1$
which the example $ f(x)=(1-x)^{-1/2} $ illustrates. The
technique which we have used to prove Theorem~\ref{maintheorem}
for integers $ k $ cannot be applied in the general real case
since it much depends on an explicit calculation of the integral
$ J_k $ in (\ref{J1}). However, numerical experiments indicate
that $ J_k $ is CM as a function of $ T=t+t^{-1} $ for all $ k>0.$
An important problem for the future is to prove that so is the
case. 

Let $ \cal A $ and $ \cal B $ denote classes of probability
distributions. We denote by ${\cal A \times \cal B}$ the class of
distributions generated by $ Y\!\cdot\!X $ for $ Y \sim \cal A $
and $ X \sim \cal B $ with $ Y $ and $ X $ independent.
Theorem~\ref{maintheorem} and Proposition~\ref{prop:GGCproducts}
can then be formulated as  $ \textrm{Gamma}(k) \times
\HM_k \subseteq \GGC$ \, and \, $\GGC
\times \GGC \subseteq \GGC,$ respectively.

One may wonder about the largest class $ {\cal H}_k $ such that $
\textrm{Gamma}(k) \times {\cal H}_k \subset \GGC. $
Apparently, because of  Theorem~\ref{maintheorem}  and
Proposition~\ref{prop:GGCproducts}, $ {\cal H}_k \supset
 \HM_k \times \GGC.$ One may also wonder about
the largest class $ {\cal G}_k $ such that ${\cal G}_k \times
\HM_k \subseteq \GGC.$  Of course, $ {\cal G}_k
\supseteq \textrm{Gamma}(k) \times \GGC. $ Could possibly
$  {\cal G}_k  \supset   \GGC(k) , $ where $
\GGC(k) $ denotes all $\GGC$s with left-extremity 0 and
total $ U$-measure at most $ k?$ To prove this possible result,
it suffices to show that $ Y \cdot X \sim \GGC$ when $ X
\sim \textrm{HM}_k $ and $ Y $ is a finite sum of independent
gamma variables with a shape parameter sum not exceeding $ k$ and varying scale parameters.

\section*{Acknowledgements}

This collaboration started when L. Bondesson was visiting A.
Behme at Technische Universit\"at Dresden. Both authors thank the
institute in Dresden for hospitality and financial support. T.\
Sj\"odin in Ume\aa \: is thanked for great interest in this work. Further, an anonymous referee is thanked for his/her effort.

\end{document}